\documentclass[12pt,a4,reqno]{amsart}
\usepackage{amsmath, amssymb}
\usepackage{graphicx}
\usepackage{subfigure}
 \usepackage{mathrsfs}
 \usepackage{bbm}
\usepackage{color}
\usepackage{esint}
\usepackage[svgnames]{xcolor} 
\usepackage[colorlinks,citecolor=red,pagebackref,hypertexnames=false,breaklinks]{hyperref}
\usepackage{pgf,tikz}
\usepackage{pdfsync}
\usepackage[bottom]{footmisc}
\usepackage{dsfont}
\usepackage{amssymb}
\usepackage{amsthm}
\newtheorem{theorem}{Theorem}
\newtheorem{remark}[theorem]{Remark}
\newtheorem{definition}[theorem]{Definition}

\newtheorem{proposition}[theorem]{Proposition}
\renewcommand{\div}{\mathrm{div}\,}

\begin{document}

\title[An improved spectral inequality]{An improved  spectral inequality  for sums of eigenfunctions}
%%%\subtitle{Untertitel / Subtitle} % if needed
\author{Axel Osses, \; \; Faouzi Triki}
\thanks{Axel Osses(Departamento de Ingenier\'{\i}a Matem\'atica, Center for Mathematical Modeling. Universidad de Chile, Chile; axosses@dim.uchile.cl)} 
\thanks{Faouzi Triki(Laboratoire Jean Kuntzmann, Universit\'e de Grenoble-Alpes, France; faouzi.triki@univ-grenoble-alpes.fr)}%{0000-0000-0000-0000}

\maketitle

\begin{abstract}
We establish a new spectral inequality for the quantified estimation of the $H^s$-norm, $s\ge 0$ of a finite linear combination of eigenfunctions in a domain in terms of its $H^s$-norm in a strictly open subset of the whole domain. The corresponding upper bound depends exponentially on the square root of the frequency number associated to the linear combination.
\end{abstract}
\vspace{0.5cm}

\section{Introduction}
Let $\Omega\subset{\mathbb R}^d, d\geq 2$, be a $C^2$ bounded domain 
with boundary $\partial\Omega$. Let us consider the eigenvalues and eigenfunctions $\{(\lambda_k, \varphi_k)\}_{k \geq 1}$ of the elliptic operator:
\begin{eqnarray}\label{eigen}
        -\div(\gamma\nabla \varphi_k) &=& \lambda_k \varphi_k,\text{ in }\Omega,\\
        \varphi_k &=& 0,\text{ on }\partial \Omega,
\end{eqnarray}
for $\gamma\in C^2(\overline \Omega)$ with $\gamma(x)\ge\gamma_0>0$  in $\Omega$, and where $\varphi_k\in H^1_0(\Omega)$ form an orthonormal basis in $L^2(\Omega)$ with a non decreasing sequence of eigenvalues. We will consider only a finite number of eigenvalues with maximal value $\lambda$ (hence until an index $K$ including eventually the repeated multiplicities):
\begin{equation}\label{eigenvalues}
0<\lambda_1\le\lambda_2\le\ldots\le\lambda_k\le\lambda_K=\lambda.
\end{equation}
We recall the following spectral inequality for a finite linear combination of eigenfunctions of Lebeau and Jerison:
\begin{theorem}[\cite{Lebeau-Jerison}]\label{thm:Lebeau-Jerison} 
        Given an nonempty open set $\omega\subset\Omega$, $\omega\neq\Omega$, there exist constants $C_1>0$ and $C_2\ge 0$ such that, if we define
    \begin{equation*}
        \psi(x)=\sum_{\lambda_k\le\lambda}a_k\varphi_k(x),\quad x\in\Omega,
    \end{equation*}
    for some coefficients $a_k\in{\mathbb R}$, then
    \begin{equation}\label{eq:Lebeau-Jerison}
        \|\psi\|_{L^2(\Omega)}\le C_1\,e^{C_2\sqrt{\lambda}}\|\psi\|_{L^2(\omega)}.
    \end{equation}

\end{theorem}
The original proof  of \cite{Lebeau-Jerison} follows an approach similar to \cite{Lin1991}, and is based on the fact that the
function $u(x,t)=\frac{1}{\sqrt{\lambda_k}}\sinh(\sqrt{\lambda_k}t)\varphi_k(x)$ is the unique solution of the elliptic equation
$u_{tt}+\div(\gamma\nabla u)=0$ in $\Omega\times(-T,T)$ for $T>0$ with 
homogeneous Dirichlet boundary conditions on $\partial\Omega\times(-T,T)$
and Cauchy data $u(x,0)=0$ and $u_t(x,0)=\psi$ on $\Omega$. So the inequality \eqref{eq:Lebeau-Jerison} is in fact the quantified unique continuation from 
a neighborhood  $\omega\times\{0\}$.\\

The spectral inequality \eqref{eq:Lebeau-Jerison} is somehow a  quantification of the unique continuation of a finite sum of 
eigenfunctions from $\omega$ to the whole domain. The exponential dependency in terms of the highest eigenvalue 
$\sqrt \lambda$ shows that the finite sum of eigenfunctions surprisingly behaves as a solution of an elliptic  PDE. Indeed previously Donnelly 
\cite{Donnelly} in a first result has used the fact that $\psi$ satisfies:
\begin{align*}
\left(\div(\gamma\nabla \cdot) + \lambda_1 \cdot \right)\left(\div(\gamma\nabla \cdot) + \lambda_2 \cdot \right)
\cdots \left(\div(\gamma\nabla \cdot) + \lambda_K \cdot \right) \psi = 0.
\end{align*}
Even if the inequality \eqref{eq:Lebeau-Jerison} is asymptotically optimal (see Theorem \ref{thm:counterexample}), there are still some interesting open problems. One important one is how the constants depend on $\omega$ \cite{Nguyen}. For instance, in the very simple case of a single frequency $\psi=\varphi_k$, for some fixed $k$, it is well known that we can take the constant $C_2=0$ in \eqref{eq:Lebeau-Jerison} provided that $\omega$ satisfies the geometrical control conditions (GCC) (see  \cite{Zuazua} and references therein). Indeed, if we define $u(x,t)=\frac{1}{\sqrt{\lambda_k}}\sin(\sqrt{\lambda_k}t)\varphi_k(x)$, we see that $u$ is the unique solution of the wave equation:
\begin{eqnarray*}
    u_{tt}-\div(\gamma\nabla u)&=&0\text{ in }\Omega\times(0,T)\\
    u(x,0)&=&0\text{ in }\Omega\\
    u_t(x,0)&=&\varphi_k\text{ in }\Omega\\
    u&=&0\text{ on }\partial\Omega\times(0,T).
\end{eqnarray*}
From the internal observability inequality for the wave equation with  $\omega$, we have that there 
exists a observability time $T_0=T_0(\omega,\Omega, \gamma, d)>0$, and a 
constant $C_0=C_0(T,\omega,\Omega, \gamma, d)>0$ such that for all $T\ge T_0$:
\begin{equation}
    \|(u(\cdot,0),u_t(\cdot,0))\|_{H^1(\Omega)\times L^2(\Omega)}\le C_0
    \|u_t\|_{L^2(\omega\times(0,T))},
\end{equation}
and since $u_t=\cos(\sqrt{\lambda_k}t)\varphi_k$ we easily obtain
for $\widetilde C_0=C_0\sqrt{T_0}$ that
\begin{equation}
    \|\varphi_k\|_{L^2(\Omega)}\le \widetilde C_0\|\varphi_k\|_{L^2(\omega)}.
\end{equation}
This shows that in contrast with sums of two and more eigenfunctions a single 
eigenfunction behaves as a  solution of a PDE of hyperbolic type. Moreover when
the geometrical control conditions (GCC) are not fulfilled and  a
weak observability controllability holds,  the optimal spectral inequality for one 
eigenfunction can be derived following the approach developed recently 
in \cite{Ammari-Triki}.\\

The other open problem is that the bound in $\sqrt{\lambda}$ in \eqref{eq:Lebeau-Jerison} is uniform for all the possible finite linear combinations. The objective of this paper is to establish an improved  non-uniform 
bound that depends on the frequency of the linear combination \cite{Ammari-Triki}. In fact the 
inequality \eqref{eq:Lebeau-Jerison} seems to be not suited for example for the sum $\phi_1+2^{-\lambda_K} \phi_K$ with a frequency that
tends to $\lambda_1$ for large $ \lambda=\lambda_K$. This will not be in contradiction with the 
spectral inequality \eqref{eq:Lebeau-Jerison} which is in fact asymptotically optimal for large $ \lambda$. Indeed, there is the following counterexample due to Lebeau and Jerison \cite{Lebeau-Jerison}. 
For the sake of completeness, we recall here the proof.
\begin{theorem}[\cite{Lebeau-Jerison}]\label{thm:counterexample} 
    Given an nonempty open set $\omega\subset\Omega$, $\omega\neq\Omega$, there
    exists $C>0$, such that for all $\lambda\ge \lambda_1$ there exist $a_k\in{\mathbb R}$ such that
    \begin{equation*}
        \psi(x)=\sum_{\lambda_k\le\lambda}a_k\varphi_k(x),\quad x\in\Omega,
    \end{equation*}
    satisfies
    \begin{equation}
        \|\psi\|_{L^2(\Omega)}\ge Ce^{C\sqrt{\lambda}}\|\psi\|_{L^2(\omega)}.
    \end{equation}
\end{theorem}

\begin{proof} Let $y_0\in\Omega\setminus\overline\omega$ be fixed, and let
$p_{y_0}(x,t)=\sum_{k=1}^\infty\ e^{-\lambda_k t}\varphi_k(x)\varphi_k(y_0)$ be the kernel in $\Omega$ of the heat equation $p_t-\div(\gamma\nabla p)=0$ with homogeneous Dirichlet boundary conditions and with an initial condition the Dirac mass at the point $y_0$. By the Maximum principle $p_{y_0}\le P_{y_0}$ where $P_{y_0}(x,t)=\frac{1}{(4\pi t)^{d/2}}\exp\left(-\frac{|x-y_0|^2}{4t}\right)$ is the corresponding kernel of the heat equation in the whole space ${\mathbb R}^d$,
and therefore
\begin{equation*}
p_{y_0}(x,t)\le c_1\exp\left(-\frac{c_2}{t}\right) \text{ for all }x\in\omega,\; t>0,
\end{equation*}
where $c_1=c/d_0^d$, $c_2=d_0^2/8$, $ c= \|r^{-d/2}\exp(-r/8)\|_{L^\infty(\mathbb R_+)}>0$,  and 
$d_0>0$ is the distance from $y_0$ to $\omega$. 
If we choose
\begin{equation}\label{eq:ak}
\psi=\sum_{\lambda_k\le\lambda} a_k\varphi_k,\quad\text{with}\quad a_k=\exp\left(-\frac{\lambda_k}{\sqrt{\lambda}}\right)\varphi_k(y_0)
\end{equation}
then
\begin{eqnarray}\label{eq:inequality3}
    \|\psi\|_{L^2(\omega)}^2&\le& 2\left\|p_{y_0}\left(\frac{1}{\sqrt{\lambda}},\cdot\right)-\psi\right\|_{L^2(\omega)}^2+2\left\|p_{y_0}\left(\frac{1}{\sqrt{\lambda}},\cdot\right)\right\|_{L^2(\omega)}^2\cr
    &\le&2\sum_{\lambda_k>\lambda}a_k^2+2c_1\,e^{-2c_2\sqrt{\lambda}}.
\end{eqnarray}
On other hand we have \cite{davies1989heat}
\begin{equation} \label{upperboundeigenfunctions}
\varphi_k(x) \leq c_3 (\sqrt\lambda_k)^{d/2+1},  \quad \forall x\in \Omega.  
\end{equation}
for some positive constant $c_3$ depending only on $\Omega$.\\

For $\lambda_k>\lambda$ and from \eqref{eq:ak} and \eqref{upperboundeigenfunctions} we have that 
\begin{equation*}
\sum_{\lambda_k>\lambda} a_k^2\le \,c_3 e^{-\sqrt{\lambda}}
\sum_{\lambda_k>\lambda} e^{-\sqrt{\lambda}_k}\lambda_k^{\frac{d}{2}+1}\le c_4\,e^{-c_5\sqrt{\lambda}}.
\end{equation*}
for some positive constant $c_4$ and then
from \eqref{eq:inequality3}
\begin{equation*}
    \|\psi\|_{L^2(\omega)}\le c_5e^{-c_6\sqrt{\lambda}},
\end{equation*}
for some constants $c_5>0$ and $c_6>0$. Finally
\[
\|\psi\|_{L^2(\Omega)}^2=\sum_{\lambda_k\le\lambda}a_k^2\ge e^{-2\sqrt{\lambda_1}}|\varphi_1(y_0)|^2>0,
\] 
and hence
\begin{equation}
\|\psi\|_{L^2(\omega)}\le \frac{c_5}{|\varphi_1(y_0)|}e^{\sqrt{\lambda_1}}
e^{-c_6\sqrt{\lambda}}\|\psi\|_{L^2(\Omega)}.
\end{equation}
\end{proof}

\section{Main results}
We will introduce the frequency numbers for a linear 
combination of eigenfunctions given by its coefficients $a_k$. 
This is the main definition to obtain a spectral inequality 
with a non-uniform exponential bound that depends on the
coefficients of the linear combination.

\begin{definition}\label{def:frequencynumber} 
Given $\tau>0$ and $n\in{\mathbb N}$, $n\ge 1$,
we define the frequency number of order $n$ associated to the coefficients $
a_k\in{\mathbb R}$ for the indices $k$ such that $\lambda_k\le\lambda$ by
\begin{equation}
\Lambda_n(\tau)=\frac{\displaystyle\sum_{\lambda_k\le \lambda}\lambda_k^n\,a_k^2\exp(2\tau\lambda_k)}{\displaystyle\sum_{\lambda_k\le \lambda}\lambda_k^{n-1} a_k^2\exp(2\tau\lambda_k)}.
\end{equation}
\end{definition}

\begin{proposition}\label{prop:frequencynumber} 
Given coefficients $a_k\in{\mathbb R}$ for $\lambda_k\le\lambda$ then the associated spectral numbers $\Lambda_n(\tau)$ form a non-decreasing sequence in $n$ of non-decreasing functions of $\tau>0$, such that
\begin{equation}\label{spectralnumbers}
0<\lambda_1<\Lambda_1(\tau)\le\Lambda_2(\tau)\le\ldots\le\lambda.
\end{equation}
\end{proposition}

\begin{proof} Given a fixed collection of coefficients 
$a_k\in{\mathbb R}$ for $\lambda_k\le\lambda$, and $n\in{\mathbb N}$ with 
$n\ge 1$, on one hand, by \eqref{eigenvalues}, it is direct to see that 
\begin{equation}
0<\lambda_1\le \Lambda_n(\tau)\le\lambda.
\end{equation}
On the other hand, if we take the derivative with respect to $\tau>0$, it is easy to check that (all the sums are over $\lambda_k\le\lambda$)
\begin{eqnarray*}
\Lambda_n'(\tau)&=&2\frac{\left(\sum \lambda_k^{n-1}\,a_k^2 e^{2\tau\lambda_k}\right)\left(\sum \lambda_k^{n+1}\,a_k^2\,e^{2\tau\lambda_k}\right)
-\left(\sum \lambda_k^n\,a_k^2\,e^{2\tau\lambda_k}\right)^2}{\left(\sum \lambda_k^{n-1}\,a_k^2\,e^{2\tau\lambda_k}\right)^2}\\
&=&2(\Lambda_n(\tau)\Lambda_{n+1}(\tau)-\Lambda_n^2(\tau)),
\end{eqnarray*}
and then we have the relationship
\begin{equation}\label{recursive-Lambda}
\Lambda_n'(\tau)=2\Lambda_n(\tau)(\Lambda_{n+1}(\tau)-\Lambda_n(\tau)).
\end{equation}
Notice that
\begin{eqnarray*}
\Lambda_n'(\tau)&=&2\left(
\frac{\sum \lambda_k^2 \lambda_k^{n-1}\,a_k^2\,e^{2\tau\lambda_k}}
{\sum \lambda_k^{n-1}\,a_k^2\,e^{2\tau\lambda_k}}-
\left(\frac{\sum \lambda_k\lambda^{n-1}\,a_k^2\,e^{2\tau\lambda_k}}{\sum \lambda_k^{n-1}\,a_k^2\,e^{2\tau\lambda_k}}\right)^2\right)
\\
&=&2\left(\sum \alpha_k\lambda_k^2-\left(\sum\alpha_k\lambda_k\right)^2\right),\quad
\alpha_k=\frac{\sum \lambda_k^{n-1}a_k^2\,e^{2\tau\lambda_k}}{\sum a_k^2\,e^{2\tau\lambda_k}},
\end{eqnarray*}
and from Jensen's inequality, since the function $x\rightarrow x^2$ is convex
and $\sum\alpha_k=1$, $\alpha_k\ge 0$, we deduce that
\begin{equation}
\Lambda_n'(\tau)\ge 0,
\end{equation}
from which we conclude that $\Lambda_n$ is a non-decreasing function of $\tau$ and thanks to \eqref{recursive-Lambda} for all $\tau>0$
\begin{equation}
\Lambda_{n+1}(\tau)\ge\Lambda_n(\tau).
\end{equation}

\end{proof}

\begin{remark}
There is an alternative way to prove that $\Lambda_1$ is an increasing function of $\tau$, see \cite{Lin2003}. Indeed, using that $u(x,t)=\sum a_k e^{(\tau-t)\lambda_k}\phi_k(x)$ is solution of the heat equation $u_t-\div(\gamma\nabla u)=0$ with homogeneous Dirichlet boundary conditions in $\Omega\times(0,\tau)$ and initial condition $u_0(x)=\sum a_k e^{\tau\lambda_k}\phi_k(x)$, it is possible to show that the function 
\[
N(t)=\frac{\|\nabla u(\cdot,t)\|_{L^2(\Omega)}^2}{\|u(\cdot,t)\|_{L^2(\Omega)^2}},
\]
is non-increasing in $t$ which is equivalent to say that $\Lambda_1(\tau)=N(0)$ is non-decreasing in $\tau$. This property also characterizes
the log-convexity of the heat equation.
\end{remark}

The following is the main result of the paper.
%%%%%%%%%%%%%%%%%%%%%%%%%%%%%%
\begin{theorem}\label{thm:Lambda} 
For all non-empty open subset $\omega\subset\Omega$, 
there exist constants $C_i=C_i(\Omega,\omega, \gamma, d) >0, \; i=1,2,3, $ such that 
for all linear combinations 
\begin{equation}
\psi(x) = \sum_{\lambda_k\le\lambda} a_k\varphi_k(x),\quad\text{for all } x\in\Omega,
\end{equation}
there exists $\tau_0\in (0, \frac{C_3}{\sqrt{\lambda_1}})$ such that
\begin{equation}\label{Lambda-ineq1}
\|\psi\|_{L^2(\Omega)}\le C_1\,e^{C_2\sqrt{\Lambda_1(\tau_0)}}
\|\psi\|_{L^2(\omega)}.
\end{equation}
Moreover, for all $s\ge 0$ there exists $\tau_s \in (0, \frac{C_3}{\sqrt{\lambda_1}})$ such that
\begin{equation}\label{Lambda-ineq2}
|\psi|_{H^s(\Omega)}\le C_1\,e^{C_2\sqrt{\Lambda_{1+s}(\tau_s)}}
|\psi|_{H^s(\omega)},
\end{equation}
where $|\psi|_{H^s(A)}=\|\sum_{\lambda_k\le\lambda} \lambda_k^{\frac{s}{2}}\,a_k\varphi_k\|_{L^2(A)}$ for $A\subseteq \Omega.$
\end{theorem}
%%%%%%%%%%%%%%%%%%%%%%%%%%%%%%%
\begin{proof} First notice that the previous calculations for $n\in{\mathbb N}$ can be easily extended replacing $n$ by a real number $\eta\ge 1$. For the heat equation for any $\tau>0$
\begin{eqnarray}
u_t-\div(\gamma\nabla u)&=&0 \quad\text{  in  }\Omega\times(0,\tau),\\
u&=&0 \quad \text{  on  }\partial\Omega\times(0,\tau),
\end{eqnarray}
and given $a_k\in{\mathbb R}$ and $\eta\ge 1$, 
let us take as initial condition
\begin{equation}
u(x,0)=\sum_{\lambda_k\le\lambda}\lambda_k^{\frac{\eta-1}{2}}
a_k\,e^{\tau\lambda_k}\varphi_k(x),
\end{equation}
so the unique solution is given by
\begin{equation}
u(x,t)=\sum_{\lambda_k\le\lambda}\lambda_k^{\frac{\eta-1}{2}}a_k\,e^{(\tau-t)\lambda_k}\varphi_k(x).
\end{equation}
We know from Theorem 4.1 in \cite{bardos2017observation} that given $\omega\subset\Omega$ there exist $C=C(\Omega,\omega, \gamma, d)>0$ and $\theta=\theta(\Omega,\omega, \gamma, d)\in(0,1)$ such that for all $\tau>0$
\begin{equation}
\|u(\cdot,\tau)\|_{L^2(\Omega)}\le e^{C\left(1+\frac{1}{\tau}\right)}
\|u(\cdot,\tau)\|_{L^2(\omega)}^\theta \|u(\cdot,0)\|_{L^2(\Omega)}^{1-\theta}.
\end{equation}
Therefore
\begin{equation*}
\left\|\sum \lambda_k^{\frac{\eta-1}{2}}
a_k\varphi_k \right\|_{L^2(\Omega)}\le e^{C\left(1+\frac{1}{\tau}\right)}
\left\|\sum \lambda_k^{\frac{\eta-1}{2}}a_k\varphi_k \right\|_{L^2(\omega)}^\theta 
\left(\sum \lambda_k^{\eta-1}a_k^2\,e^{2\tau\lambda_k}\right)^{\frac{1-\theta}{2}},
\end{equation*}
that is for $s=\eta-1$
\begin{equation}\label{eq:inequality1}
|\psi|_{H^s(\Omega)}\le e^{C\left(1+\frac{1}{\tau}\right)}
|\psi|_{H^s(\omega)}^\theta \left(\phi(\tau)\right)^{\frac{1-\theta}{2}}.
\end{equation}
where
\begin{equation}    \phi(\tau)=\sum \lambda_k^{\eta-1}a_k^2\,e^{2\tau\lambda_k}.
\end{equation}
Now
\begin{eqnarray*}
    \phi'(\tau)&=&2\sum \lambda_k^\eta a_k^2\,e^{2\tau\lambda_k}\\ 
    &=&2\frac{\sum \lambda_k^\eta a_k^2\,e^{2\tau\lambda_k}}{\phi(\tau)}\phi(\tau)\\ 
    &=&2\Lambda_\eta(\tau)\phi(\tau),
\end{eqnarray*}
where $\Lambda_\eta(\tau)$ was introduced in Definition \ref{def:frequencynumber}. Therefore
\begin{equation}
\phi(\tau)=\phi(0)e^{2\int_0^\tau \Lambda_\eta(r)dr},\quad
\phi(0)=|\psi|_{H^s(\Omega)}^2.
\end{equation}
By replacing in \eqref{eq:inequality1} gives
\begin{equation}\label{eq:inequality2}
|\psi|_{H^s(\Omega)}\le e^{c_1}
|\psi|_{H^s(\omega)}e^{\frac{c_1}{\tau}+c_2\int_0^\tau\Lambda_\eta(r) dr}.
\end{equation}
where $c_1=C/\theta$ and $c_2=(1-\theta)/\theta$. The constants $c_1$ 
and $c_2$ depend only in $\Omega, \omega, \gamma, d$,  and are independent of $\tau$.

Let
\begin{equation}
    f(\tau)=\frac{c_1}{\tau}+c_2\int_0^\tau \Lambda_\eta(r) dr,
\end{equation}
so
\begin{equation}
    f'(\tau)=-\frac{c_1}{\tau^2}+c_2\Lambda_\eta(\tau).
\end{equation}
By optimizing in $\tau$ since $1/\tau^2$ is strictly decreasing to zero
and $\Lambda_\eta(\tau)$ is a non decreasing in $\tau$ (see Proposition \ref{prop:frequencynumber}) and not identically zero, we obtain that
there exist a $\tau_s>0$ solution of the equation
\begin{equation} \label{tauS}
\tau_s=\frac{\sqrt{c_1/c_2}}{\sqrt{\Lambda_\eta(\tau_s)}},\quad c_1/c_2=C/(1-\theta),
\end{equation}
and since $\Lambda_\eta$ is increasing in $\tau$ we have
\begin{eqnarray*}
f(\tau_s)&=&\frac{c_1}{\sqrt{c_1/c_2}}\sqrt{\Lambda_\eta(\tau_s)}+c_2\int_0^{\tau_s}\Lambda_\eta dr\\
&\le&\frac{c_1}{\sqrt{c_1/c_2}}\sqrt{\Lambda_\eta(\tau_s)}+c_2\,\tau_s\,\Lambda_\eta(\tau_s)\\
&\le& c_3\sqrt{\Lambda_\eta(\tau_s)},
\end{eqnarray*}
where $c_3=2c_1/\sqrt{c_1/c_2}=\frac{2}{\theta}\sqrt{C(1-\theta)}$. Replacing in 
\eqref{eq:inequality2} we obtain
\begin{equation}\label{inequality4}
|\psi|_{H^s(\Omega)}\le e^{c_1}e^{c_3\sqrt{\Lambda_\eta(\tau_s)}}
|\psi|_{H^s(\omega)},
\end{equation}
and we conclude the proof with $C_1=e^{c_1}$ and $C_2=c_3$ after replacing $\eta=1+s$. \\

Again since $\tau \to \Lambda_\eta(\tau)$ is non-decreasing, we deduce from  \eqref{tauS}, that
\[ \tau_s \leq \frac{\sqrt{c_1/c_2}}{\sqrt{\Lambda_\eta(0)}} = \frac{\sqrt{c_1/c_2}}{\sqrt{\lambda_1}}.
\]
Taking finally $C_3 = \sqrt{c_1/c_2}$ finishes the proof of the Theorem.
\end{proof}
\begin{remark}
    Notice that
    \begin{equation}
    \Lambda_{1+s}(\tau_s)\le \lambda,       
    \end{equation}
so we recover from the previous result the original inequality of Lebeau-Jerison \cite{Lebeau-Jerison}, but the bound in terms of the frequency number is more precise
since it depends on the coefficients of the linear combination of the eigenfunctions.
\end{remark}

\begin{remark}\label{rem:both}
    We see in the proof of Theorem \ref{thm:Lambda} 
    that the constants $C_1$ and $C_2$ appearing in
    \eqref{Lambda-ineq1} and \eqref{Lambda-ineq2}
    are the same, and they are in fact the same constants as the ones appearing in Theorem \ref{thm:Lebeau-Jerison} in \eqref{eq:Lebeau-Jerison}. Indeed, 
    if we bound $\lambda_k\le\lambda$ uniformly, then from \eqref{eq:inequality1} we obtain:
    \begin{equation*}
    |\psi|_{H^s(\Omega)}\le e^{\frac{C}{\theta}\left(1+\frac{1}{\tau}\right)}e^{\tau\lambda\frac{1-\theta}{\theta}}
    |\psi|_{H^s(\omega)}.
    \end{equation*}
    so after optimizing in $\tau$ we obtain 
    $\tau=\sqrt{\frac{C}{1-\theta}}\frac{1}{\sqrt{\lambda}}$, from which
    we obtain that
    \begin{equation}
|\psi|_{H^s(\Omega)}\le e^{c_1}e^{c_3\sqrt{\lambda}}
|\psi|_{H^s(\omega)},
\end{equation}
with the same constants $c_1=C/\theta$ and $c_3=\frac{2}{\theta}\sqrt{C(1-\theta)}$ of \eqref{inequality4}.
\end{remark}
%\begin{remark}
%    Say why we are not violating the counterexample of Lebeau-Jerison.
%    Idea... we could use the coefficients of the function $\psi(r)=%\exp\left(\frac{1}{r_1^2-r^2}\right)$ for $r\le r_1$ and zero for $r_1<r<r_2$ with $\Omega=B(0,r_2)$, $0<r_1<r_2$ and $\omega=B(0,r_2)\setminus \overline{B(0,r_1)}$, which is $C^\infty(\overline\Omega)$ but not an analytic function and it should be non observable. 
%    Or directly, take the heat kernel type function of the counterexample of Lebeau-Jerison.
%\end{remark}

\begin{remark}
Since $ \Lambda_1(\tau) \leq \lambda$ for all $\tau \in \mathbb R_+$, we deduce from Theorem \ref{thm:counterexample} that for any given an nonempty open set $\omega\subset\Omega$, $\omega\neq\Omega$, there
    exists $C>0$, such that for all $\lambda \ge \lambda_1$ there exist $a_k\in{\mathbb R}$ such that
    \begin{equation*}
        \psi(x)=\sum_{\lambda_k\le\lambda}a_k\varphi_k,\quad x\in\Omega,
    \end{equation*}
    satisfies
    \begin{equation}
        \|\psi\|_{L^2(\Omega)}\ge Ce^{C\sqrt{\Lambda_1(\tau)}}\|\psi\|_{L^2(\omega)},
    \end{equation}   

    for all $\tau \in \mathbb R_+$. This shows that the estimate \eqref{Lambda-ineq1} is also asymptotically optimal.
\end{remark}
\begin{remark}
Notice that the estimate \eqref{Lambda-ineq1}  in Theorem \ref{thm:Lambda} is independent of $K$, the largest index of the eigenvalues with 
eigenfunctions appearing in the finite 
sums. Therefore   it is straightforward to verify that it remains  valid 
for infinite sums  if in addition $ \Lambda_1(\tau)<\infty$ holds for all $\tau$. 
\end{remark}

\section{Numerical tests}
%We consider an example where we compute the eigenfrequencies and eigenvalues with $\Omega=B(0,1)$ a unit circle and $\omega=B(-1/2,-1/2,1/5)$. We build a mesh for evaluate the $L^2$ norms of the linear combination \eqref{eq:ak}
%in $\Omega$ and $\omega$. We compute the first $60$ eigenvalues and their corresponding eigenvectors using Bessel functions.

We consider an example where we compute the eigenfrequencies and eigenvalues with $\Omega=B(0,1)$ a unit circle, $C=1$ and different values of the constant $\theta$. We compute the first $60$ eigenvalues and their corresponding eigenvectors using Bessel functions.

The examples we consider are the followings:
\begin{equation*}
    a_k=e^{-(\lambda_k/\sqrt{\lambda})^p},\quad p=1.0;1.2;1.3.
\end{equation*}
The case $p=1$ corresponds to coefficients that decay as in the counterexample of the proof of Theorem \ref{thm:counterexample} which is
in some extend the critical case, for which the differences between 
$\lambda$ and $\Lambda_1$ will begin to appear 
and the other cases $p=1.2$ and $p=1.3$ correspond 
to faster decays where the difference will became stronger.

In Figure \ref{fig:1} we show the intersection for the case of the example with $p=1.2$ where we see the curves $\frac{C}{(1-\theta)\tau^2}$ and $\Lambda_1$, $\Lambda_2$ and $\Lambda_3$. The optimal time $\tau_s$ for the new inequality \eqref{Lambda-ineq2} corresponds to the 
intersection of these curves. The value of $\tau_0$
is shown, which corresponds to the intersection of $\frac{C}{(1-\theta)\tau^2}$ and $\Lambda_1$ and it is compared with the optimal value 
for the inequality \eqref{Lambda-ineq1} given by $\tau=\frac{\sqrt{C(1-\theta)}}{\sqrt{\lambda}}$. 

Thanks to Remark \ref{rem:both}, for given fixed values of $C$ and $\theta$, we know that the involved constants of the inequalities 
  \eqref{Lambda-ineq1} and \eqref{Lambda-ineq2}
can be taken as the same. Therefore, we can directly compare $\sqrt{\lambda}$ and $\sqrt{\Lambda_1}$ in order to compare the estimates. 
This is exactly what is shown in Figures \ref{fig:2}, \ref{fig:3} and \ref{fig:4}. We observe in these figures that the estimation with $\sqrt{\Lambda_1}$ is better than the estimation with $\sqrt{\lambda}$ notably in the cases where the coefficients of the linear combination decay exponentially (Examples 1 and 2). When the coefficients are uniform (Example 3) there are no big differences between the bounds 
in $\sqrt{\lambda}$ and $\sqrt{\Lambda_1}$ as expected.

\begin{figure}
    \centering
    \includegraphics[width=11cm]{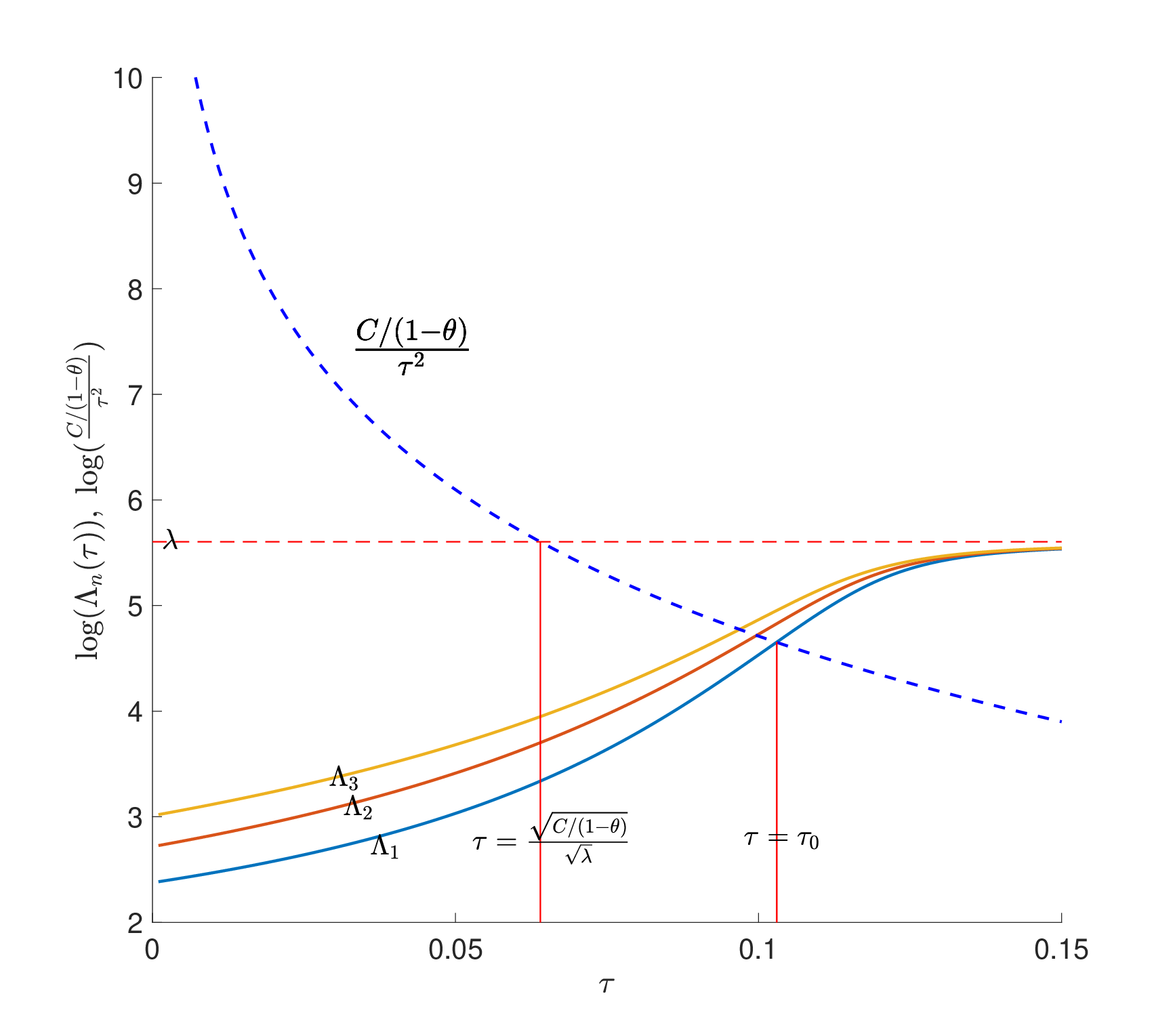}
    \caption{The intersection for the case of the Example with $p=1.2$ between $\frac{C}{(1-\theta)\tau^2}$
    and $\Lambda_\eta(\tau)$, $\eta=1+s$ gives the optimal time $\tau_s$. We show the value of $\tau_0$.}
    \label{fig:1}
\end{figure}

\begin{table}[]
    \centering
    \begin{tabular}{c|c|c|c}
      Example    &$\theta=0.1$ & $\theta=0.25$ & $\theta=0.5$ \\
          \hline
$p=1.0$ & $\tau_0=0.070$ & $\tau_0=0.075$ & $\tau_0=0.089$ \\
$p=1.2$ & $\tau_0=0.103$ & $\tau_0=0.106$ & $\tau_0=0.113$ \\
$p=1.3$ & $\tau_0=0.121$ & $\tau_0=0.126$ & $\tau_0=0.136$ \\
\hline
    \end{tabular}
    \caption{Optimal times $\tau_0$ for $\Lambda_1$ for different values of the parameter $\theta$ ($C=1$).}
    \label{tab:cases}
\end{table}

\begin{figure}
    \centering
    \includegraphics[width=11cm]{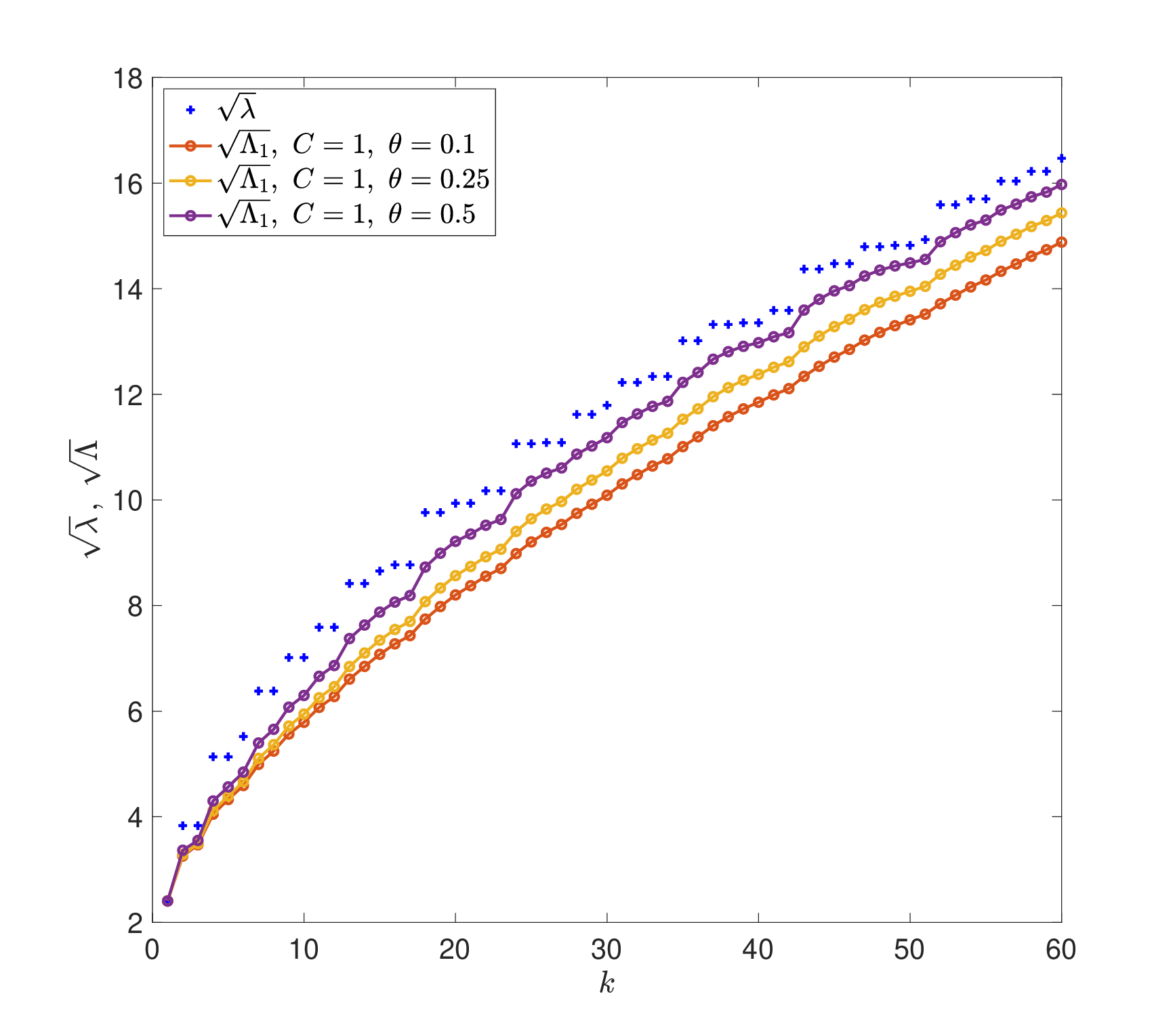}
    \caption{Comparison between $\sqrt{\lambda}$ and $\sqrt{\Lambda_1}$ for
    different choices of the constant $\theta$. Example with $p=1.0$.}
    \label{fig:2}
\end{figure}

\begin{figure}
    \centering
    \includegraphics[width=11cm]{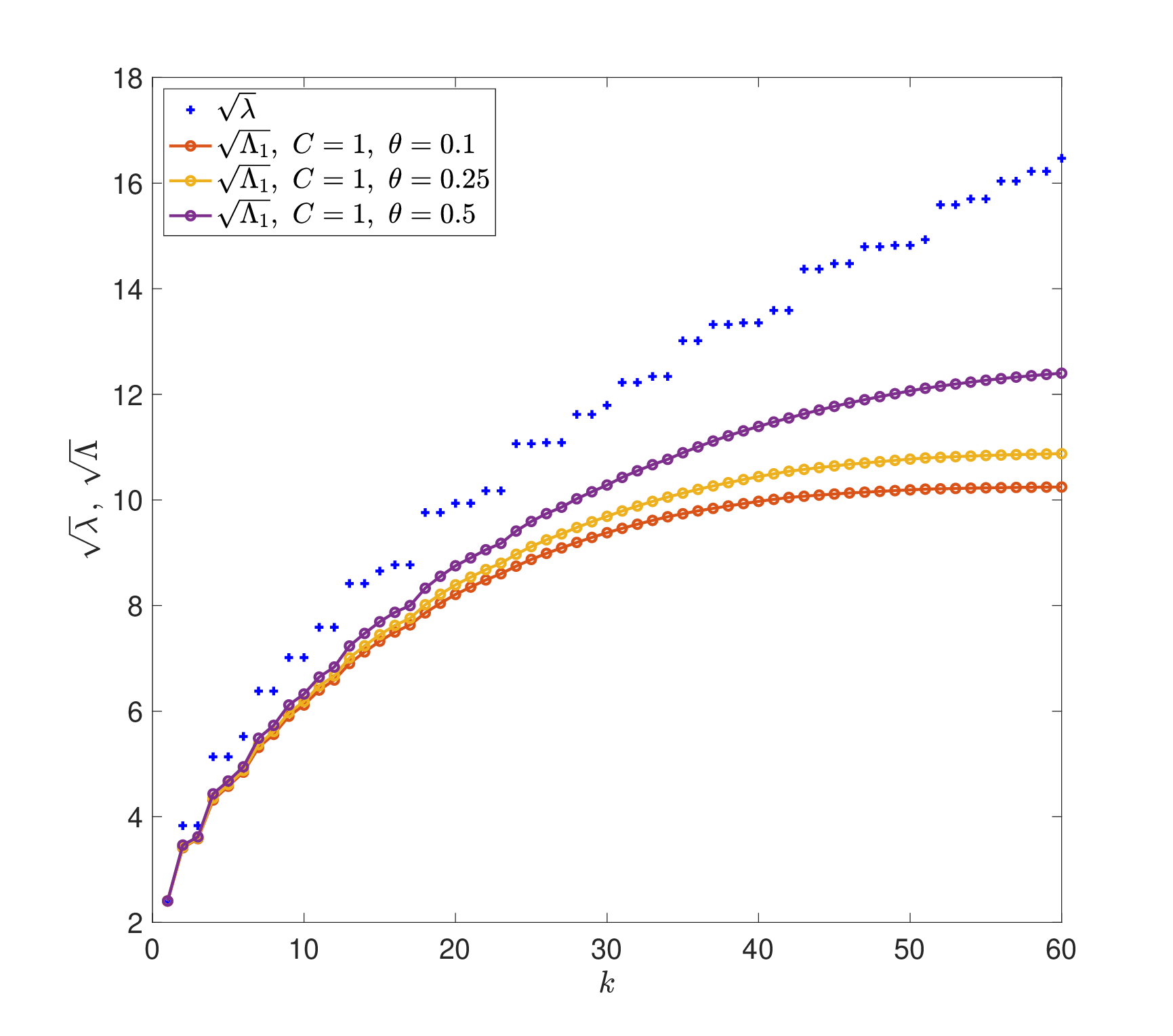}
    \caption{Comparison between $\sqrt{\lambda}$ and $\sqrt{\Lambda_1}$ for
    different choices of the constant $\theta$. Example with $p=1.2$.}
    \label{fig:3}
\end{figure}

\begin{figure}
    \centering
    \includegraphics[width=11cm]{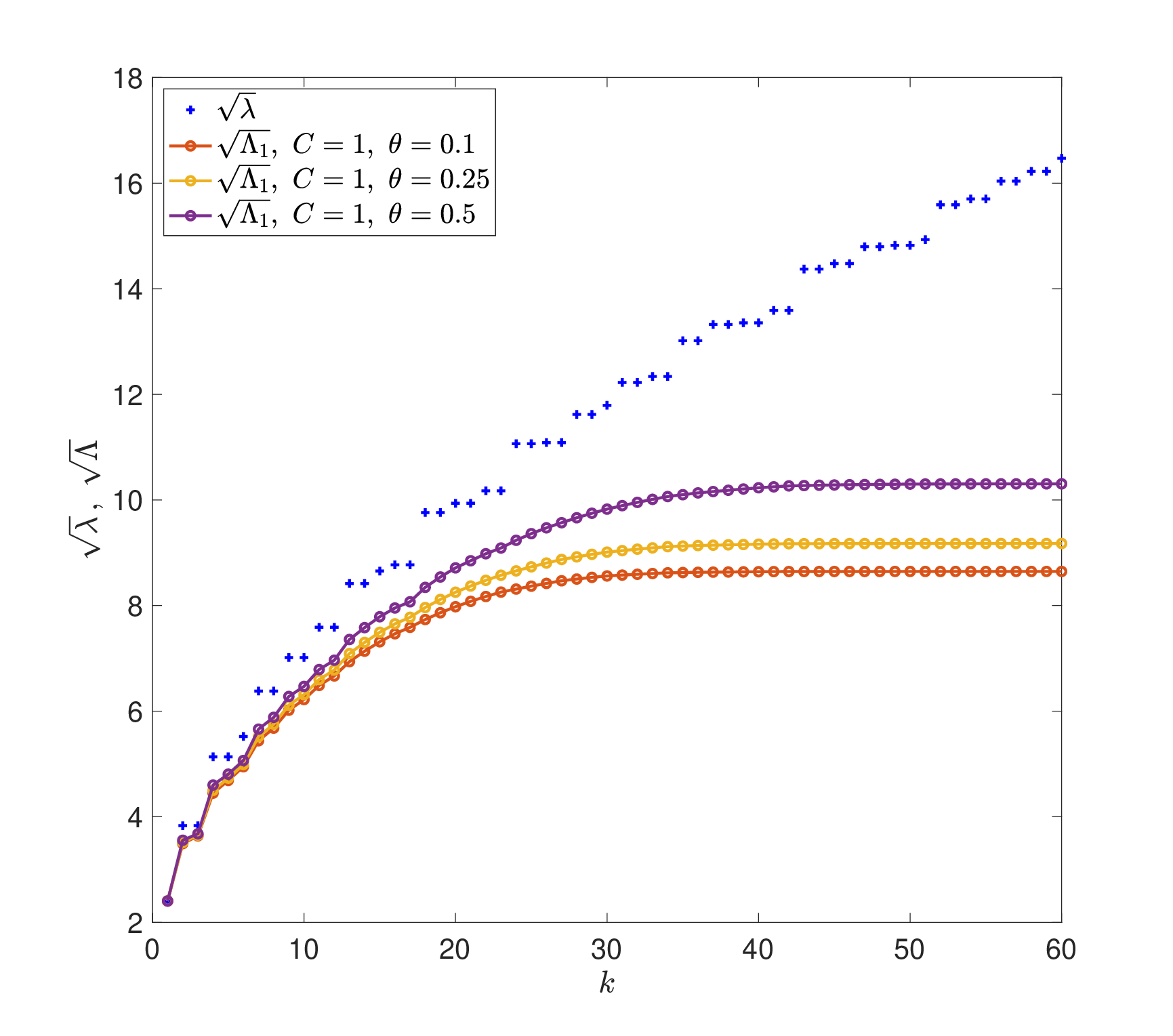}
    \caption{Comparison between $\sqrt{\lambda}$ and $\sqrt{\Lambda_1}$ for
    different choices of the constant $\theta$. Example with $p=1.3$.}
    \label{fig:4}
\end{figure}

\bibliographystyle{alpha}
\bibliography{spectral_inequality} 
%\printbibliography
\end{document}